\documentclass[a4paper,10pt]{article}

\usepackage[dvipdfm,%
 bookmarks=true,%
 bookmarksnumbered=true,%
 colorlinks=true,%
 pdftitle={},%
 pdfauthor={Yokoyama},%
 pdfsubject={},%
 pdfkeywords={}]{hyperref}

\makeatletter
\def\blfootnote{\gdef\@thefnmark{}\@footnotetext}
\makeatother

\usepackage{soul}

\usepackage{latexsym}
\usepackage{amsmath}
\usepackage{amscd}
\usepackage{amssymb}
\usepackage{authblk}
\usepackage[bbgreekl]{mathbbol}

\usepackage{url}
\usepackage{amsthm}

\def\RCAo{\mathsf{RCA_0}}

\def\WKLo{\mathsf{WKL_0}}

\def\E{\exists}
\def\A{\forall}
\def\N{\mathbb{N}}

\def\dom{\mathop{\mathrm{dom}}\nolimits}

\newcommand{\mc}[1]{\mathcal{#1}}

\def\P2{\Pi^1_2}

\def\PHt{\mathrm{PH}^2_2}
\def\RT{\mathrm{RT}}

\newcommand{\D}{\mathrm{D}}
\newcommand{\low}{\mathrm{low}}

\def\COH{\mathrm{COH}}

\def\BII{\mathrm{B}\Sigma^0_2}
\def\BIII{\mathrm{B}\Sigma^0_3}
\def\II{\mathrm{I}\Sigma^0_1}
\def\III{\mathrm{I}\Sigma^0_2}

\newcommand{\IN}[1][n]{\mathrm{I}\Sigma^0_{#1}}

\newcommand{\BN}[1][n]{\mathrm{B}\Sigma^0_{#1}}

\newcounter{menum}
{\begin{enumerate}%
\setcounter{enumi}{#1}}%
{\setcounter{menum}{\value{enumi}}\end{enumerate}}

\newtheorem{thm}{Theorem}[section]

\newtheorem*{theorem*}{Theorem}

\newtheorem*{claim*}{Claim}

\newtheorem{lem}[thm]{Lemma}

\theoremstyle{definition}
\newtheorem{defi}{Definition}[section]

\newtheorem{rem}[thm]{Remark}

\usepackage{xcolor}

\definecolor{lightred}{rgb}{1,.60,.60}

\begin{document}
\newpage

\title{The strength of Ramsey's theorem for\\ pairs and arbitrarily many colors}

\author[1]{Theodore A.~Slaman}
\author[2]{Keita Yokoyama}

\affil[1]{\small \sf{slaman@math.berkeley.edu}}
\affil[2]{\small \sf{y-keita@jaist.ac.jp}}

\date{July 5, 2018
}

\blfootnote{
The first author is partially supported by the National Science Foundation grant DMS-1600441.
The second author is partially supported by
JSPS KAKENHI (grant numbers 16K17640 and 15H03634) and JSPS Core-to-Core Program
(A.~Advanced Research Networks).
 }

\maketitle
\def\Bexp{\mathrm{B}\Sigma_{1}+\mathrm{exp}}
\newcommand\RF{\mathrm{RF}}
\newcommand\tpl{\mathrm{tpl}}
\newcommand\col{\mathrm{col}}
\newcommand\fin{\mathrm{fin}}
\newcommand\Log{\mathrm{Log}}
\newcommand\Ct{\mathrm{Const}}
\newcommand\It{\mathrm{It}}
\renewcommand\PHt{\widetilde{\mathrm{PH}}{}}
\newcommand\BME{\mathrm{BME}_{*}}
\newcommand\HT{\mathrm{HT}}
\newcommand\Fin{\mathrm{Fin}}
\newcommand\FinHT{\mathrm{FinHT}}
\newcommand\wFinHT{\mathrm{wFinHT}}
\newcommand\FS{\mathrm{FS}}
\newcommand\LL{\mathsf{L}}
\newcommand\GPg{\mathrm{GP}}
\newcommand\GP{\mathrm{GP}^{2}_{2}}
\newcommand\FGPg{\mathrm{FGP}}
\newcommand\FGP{\mathrm{FGP}^{2}_{2}}
\newcommand\SGP{\mathrm{SGP}^{2}_{2}}
\newcommand\Con{\mathrm{Con}}
\newcommand\WF{\mathrm{WF}}
\newcommand\bb{\mathbf{b}}

\definecolor{lightred}{rgb}{1,.60,.60}
\newcommand{\ted}[1]{\sethlcolor{lightred}\hl{#1}}
\newcommand{\keita}[1]{\sethlcolor{yellow}\hl{#1}}
\newcommand{\tedf}[1]{\sethlcolor{lightred}\hl{\footnote{\hl{#1}}}}
\newcommand{\keitaf}[1]{\sethlcolor{yellow}\hl{\footnote{\hl{#1}}}}

\begin{abstract}
In this paper, we will show that $\RT^{2}+\WKLo$ is a $\Pi^{1}_{1}$-conservative extension of $\BN[3]$.
\end{abstract}

\section{Introduction}
The strength of Ramsey's theorem is well-studied in the setting of reverse mathematics.
In this paper, we will focus on the first-order consequences of Ramsey's theorem for pairs over the base system $\RCAo$.
On the first-order part of Ramsey's theorem for pairs and two colors ($\RT^{2}_{2}$), Hirst\cite{Hirst-PhD} showed that it implies $\BN[2]$ and then Cholak/Jockusch/Slaman\cite{CJS} proved that $\RT^{2}_{2}+\WKLo+\IN[2]$ is a $\Pi^{1}_{1}$-conservative extension of $\IN[2]$.
Thus, its first-order part is in between $\BN[2]$ and $\IN[2]$.
There are many studies to determine the exact strength, and recently Chong/Slaman/Yang\cite{CSY2017} showed that $\RCAo+\RT^{2}_{2}$ does not imply $\IN[2]$, and Patey/Yokoyama\cite{PY2018} showed that $\WKLo+\RT^{2}_{2}$ is a $\Pi^{0}_{3}$-conservative extension of $\BN[2]$, which means that the first-order part of $\RT^{2}_{2}$ is closer to $\BN[2]$.

How about the strength of Ramsey's theorem for pairs and arbitrarily many colors ($\RT^{2}$)?
Over $\RCAo$, one may easily see that $\RT^{2}_{k}$ implies $\RT^{2}_{k+1}$, but that does not mean $\RT^{2}_{2}$ implies $\RT^{2}$ since the induction available within $\RCAo$ is not strong enough.
Indeed, the case for $\RT^{2}$ is very similar to the case for $\RT^{2}_{2}$ and the following are known.
\begin{thm}[Hirst\cite{Hirst-PhD}]
$\RT^{2}+\RCAo$ implies $\BN[3]$.
\end{thm}
\begin{thm}[Cholak/Jockusch/Slaman\cite{CJS}]\label{thm:CJS-RT2}
$\RT^{2}+\WKLo+\IN[3]$ is a $\Pi^{1}_{1}$-conservative extension of $\IN[3]$.
\end{thm}
Hence, the first-order part of $\RT^{2}$ is between $\BN[3]$ and $\IN[3]$.
Here, we will sharpen the proof of this theorem, and determine the exact first-order part of $\RT^{2}$, namely it is $\BN[3]$.
For the basic notions of this area, see \cite{CJS,Slicing-the-truth,SOSOA}.



\section{The first-order part of $\RT^{2}$}
Our main conservation theorem is the following.
\begin{thm}
 $\RT^{2}+\WKLo$ is a $\Pi^{1}_{1}$-conservative extension of $\BN[3]$.
\end{thm}
To show the main theorem, we will sharpen the argument from \cite{CJS}, which is used for the proof of Theorem~\ref{thm:CJS-RT2}.

\begin{thm}[Cholak/Jockusch/Slaman\cite{CJS}]
Over $\RCAo$, $\RT^{2}$ is equivalent to $\D^{2}$ plus $\COH$.
\end{thm}
Here, $\D^{2}$ and $\COH$ are the following statements.
\begin{itemize}
 \item[] $\D^{2}$: for any $k\in \N$ and any $\Delta^{0}_{2}$-partition $\N=\bigsqcup_{i<k} \mc{A}_{i}$, there exists an infinite set $Z\subseteq \N$ such that $Z\subseteq \mc{A}_{i}$ for some $i<k$,
 \item[] $\COH$: for any infinite sequence of sets $\langle R_{i}: i\in\N \rangle$, there exists an infinite set $Z\subseteq \N$ such that $(Z\subseteq^{*} R_{i}\vee Z\subseteq^{*} \N\setminus R_{i})$ for any $i\in\N$.
\end{itemize}
(Note that $\N$ denotes the set of all natural numbers within $\RCAo$, \textit{i.e.}, if $\mc M=(M,S)$ is a model of $\RCAo$, $\N^{\mc M}=M$.)

Since we already know that $\RCAo+\RT^{2}$ implies $\BN[3]$, we will consider the first-order strength of the above two statements over $\BN[3]$.
Note that $\D^{2}$ and $\COH$ are both $\Pi^{1}_{2}$-statements, and $\Pi^{1}_{1}$-conservation results for $\Pi^{1}_{2}$-statements can be amalgamated, \textit{i.e.}, if both of $\RCAo+\BN[3]+\D^{2}$ and $\RCAo+\BN[3]+\COH$ are $\Pi^{1}_{1}$-conservative over $\BN[3]$ then so is $\RCAo+\BN[3]+\D^{2}+\COH$, which is equivalent to $\RCAo+\RT^{2}$ (see \cite{Y2010}).
The strength of $\COH$ (together with weak K\"onig's lemma) over $\BN[3]$ is already known.
\begin{thm}[H\'ajek\cite{Hajek}, Belanger\cite{Belanger}]
$\WKLo+\COH+\BN[3]$ is a $\Pi^{1}_{1}$-conservative extension of $\BN[3]$.
\end{thm}

Thus, what we need is the following.
\begin{thm}\label{thm:main-conservation}
$\RCAo+\D^{2}+\BN[3]$ is a $\Pi^{1}_{1}$-conservative extension of $\BN[3]$.
\end{thm}
In \cite{CJS}, it is shown by a variant of Mathias forcing that a computable instance of $\D^{2}$ admits a $\low_{2}$-solution.
On the other hand, $\low_{2}$-sets preserve $\BN[3]$ since they won't add any new $\Sigma^{0}_{3}$-sets.
Thus, the following theorem is essential for Theorem~\ref{thm:main-conservation}.
\begin{thm}\label{thm:main-construction}
Let $(M,\{B\})$ be a countable model of $\BN[3]$, and let $M=\bigsqcup_{i<k} \mc{A}_{i}$ be a $\Delta^{B}_{2}$-partition of $M$ for some $k\in M$.
Then there exists an unbounded $\Delta^{B}_{3}$-set $G\subseteq M$ such that $G\subseteq \mc{A}_{i}$ for some $i<k$, and
any $\Sigma^{B\oplus G}_{3}$ subset of $M$ is already $\Sigma^{B}_{3}$ in $M$.
\end{thm}
We will prove this theorem in the next section.
Assuming this theorem, it is routine work to prove Theorem~\ref{thm:main-conservation}.
\begin{proof}[\it Proof of Theorem~\ref{thm:main-conservation}.]
Assume that $\BN[3]$ does not prove a $\Pi^{1}_{1}$-sentence $\A X\psi(X)$.
Then there exists a countable model $(M,S)\models \BN[3]$ such that $(M,S)\models \neg\psi(B)$ for some $B\in S$.
For $X,Y\subseteq M$, $X\le_{T} Y$ means that $X$ is $\Delta^{Y}_{1}$ in $M$.
By using Theorem~\ref{thm:main-construction} repeatedly, one can construct an $\omega$-length sequence of subsets of $M$, $B=B_{0}\le_{T}B_{1}\le_{T} \dots$
%
 so that
\begin{itemize}
 \item for any $m\in\omega$ and $\Delta^{B_{m}}_{2}$-partition $M=\bigsqcup_{i<k} \mc{A}_{i}$, there exist $n\ge m$ and an unbouded set $G\le_{T} B_{n}$ such that $G\subseteq \mc{A}_{i}$ for some $i<k$, and,
 \item any $\Sigma^{B_{m}}_{3}$ subset of $M$ is already $\Sigma^{B}_{3}$ in $M$.
\end{itemize}
Put $\bar S=\{X\subseteq M: X\le_{T} B_{m}, m\in\omega\}$, then $(M,\bar S)\models \RCAo+\D^{2}+\BN[3]$ but $\neg\psi(B)$ is still true in $(M,\bar S)$.
Hence $\RCAo+\D^{2}+\BN[3]$ does not prove $\A X\psi(X)$.
\end{proof}

\section{Construction}
In this section, we will prove Theorem~\ref{thm:main-construction}.
The main idea is formalizing a computability theoretic construction within a nonstandard model of arithmetic.
The following theorem is a basic tool to formalize standard arguments for $\Pi^{0}_{1}$-classes, and we will use it freely throughout this section.
\begin{thm}\label{thm:Pi01-in-WKL}
Let $\varphi(X,A)$ be a $\Pi^{0}_{1}$-formula with exactly displayed the set variables.
\begin{enumerate}
 \item There exists a $\Pi^{0}_{1}$-formula $\psi(A)$ such that $\WKLo$ proves $\E X\varphi(X,A)\leftrightarrow \psi(A)$.
 \item $\WKLo$ proves that $\E X\varphi(X,A)$ is equivalent to the statement that there exists (a $\Delta^{A}_{2}$-code for) a low set $Y$ relative to $A$ such that $\varphi(Y,A)$.
 \item For a given $\Delta^{0}_{2}$-definable set $\mc{A}$ (possibly not a second-order object), $\WKLo+\BII$ proves $\E X\varphi(X,\mc{A})\to \E X\E Y\varphi(X,Y)$.
 Thus, ``there exists $\Delta^{0}_{2}$-definable set $\mc{A}$ such that $\E X\varphi(X,\mc A)$'' can be described by a $\Pi^{0}_{1}$-formula.
\end{enumerate}
\end{thm}
\begin{proof}
1 is a well-known fact, see, e.g., \cite[Lemma~VIII.2.4]{SOSOA}.
2 is a low basis theorem for $\Pi^{0}_{1}$-classes which is formalizable within $\II$ \cite{HK1989}.
With $\BII$, one can mimic the proof of 1 for $\Delta^{0}_{2}$-sets, 3 easily follows from that.
\end{proof}

As we mentioned in the previous section, we want to formalize the second $\low_{2}$-solution construction for $\D^{2}$ from \cite{CJS} within $\BN[3]$.
However, that construction uses $\IN[3]$ in two parts,
 to find the right color for a solution, and to do $\mathbf{0}''$-primitive recursion.
In the following construction, we need to avoid these.
To overcome the first problem, we will construct solutions for all possible colors, and see that it works for at least one color in the end.
For the second problem, we will still use $\mathbf{0}''$-primitive recursion.
In a nonstandard model $(M,S)\models\BN[3]$, $\mathbf{0}''$-primitive recursion might end in nonstandard numbers of steps which form a proper cut of $M$.
Thus, we will decide some finite collection of $\Sigma^{0}_{2}$-statements at each step,
and finally decide all $\Sigma^{0}_{2}$-statements before $\mathbf{0}''$-primitive recursion ends,
 adapting Shore's blocking argument.

Now we start the construction.
Let $(M,\{B\})$ be a countable model of $\BN[3]$.
By the following theorem, we will work  within $(M,S)\models \WKLo+\BIII$ with $B\in S$.
\begin{thm}[H\'ajek\cite{Hajek}]
Let $(M,\{B\})$ be a countable model of $\BN[3]$.
Then there exists $S\subseteq \mc{P}(M)$ such that $B\in S$ and $(M,S)\models \WKLo+\BN[3]$. 
\end{thm}
In what follows, we will mimic the ``double jump control'' method in \cite{CJS}.
Let $\bigsqcup_{i<k}\mc{A}_{i}=M$ be a $\Delta^{B}_{2}$-partition for some $k\in M$ and $B\in S$.
A quintuple $p=(\bar{F}, X, \sigma, \ell_{0},\ell_{1})$ is said to be a \textit{pre-condition} if
\begin{itemize}
 \item $\ell_{0},\ell_{1}\in M$, $\sigma:\ell_{0}\times k\to 2$,
 \item $\bar{F}$ is a $k$-tuple of finite sets $\langle F_{i}: i<k \rangle$ such that $F_{i}\subseteq \mc{A}_{i}$,
 \item $X$ is coded by $\ell_{1}$ and (a $\Delta^{B}_{2}$ code for) an infinite $\low^{B}$ set $X_{0}$ as $X=X_{0}\cap(\ell_{1},\infty)$,
 \item $\max \bar{F}\cup \{\ell_{0}\}<\ell_{1}$, and a code for $X_{0}$ is bounded by $\ell_{1}$.
\end{itemize}
Here, we call a pair of $k$-tuple of finite sets and another set $(\bar{F},X)$ with $\min X>\max\bar{F}$ a Mathias pair.
(In what follows, we will mainly deal with an infinite Mathias pair, \textit{i.e.}, a Mathias pair with $X$ infinite, but quantification for Mathias pairs ranges over possibly finite Mathias pairs.)
For finite sets $E,F$ and another set $X$, we write $E\in (F,X)\leftrightarrow F\subseteq E\subseteq F\cup X$.
For two Mathias pairs $(\bar{F},X),(\bar{E}, Y)$, we say that $(\bar{E}, Y)$ \textit{extends} $(\bar{F},X)$ (write $(\bar{F},X)\ge (\bar{E}, Y)$) if $E_{i}\in (F_{i},X)$ for every $i<k$, and $Y\subseteq X$.

Next, we define how Mathias pairs force $\Sigma^{0}_{1}$ and $\Sigma^{0}_{2}$-formulas at each color.
To control the complexity of forcing formulas, we consider a triple of the form $(\bar{F},X,\ell)$, which is a Mathias pair $(\bar{F},X)$ with a bound $\ell\in M$.
Let $\theta(n, G[n])$ be a $\Sigma^{0}_{0}$-formula with a new variable $G$.
Then we define strong forcing $\Vdash^{+}$ for a pair of color $i$ and a $\Sigma^{0}_{1}$-formula $\E n\,\theta(n, G[n])$ as
$$(\bar{F},X,\ell)\Vdash^{+}\langle i,\E n\,\theta(n, G[n]) \rangle\Leftrightarrow \E n\le \max F_{i}\,\theta (n, F_{i}[n]).$$
Similarly, let $\theta(m,n, G[n])$ be a $\Sigma^{0}_{0}$-formula with a new variable $G$.
Then we define forcing $\Vdash$ for a pair of color $i$ and a $\Sigma^{0}_{2}$-formula $\E m\A n\,\theta(m,n, G[n])$ as
$$(\bar{F},X,\ell)\Vdash\langle i,\E m\A n\,\theta(m,n,G[n]) \rangle\Leftrightarrow \E m\le \ell\,\A E\in (F_{i},X)\A n\le\max E\,\theta (m,n, E[n]).$$

Let $\pi(e, m, G)\equiv \A n\, \pi_{0}(e,m,n,G[n])$ be a universal $\Pi^{B,G}_{1}$-formula, \textit{i.e.}, a universal $\Pi^{0}_{1}$-formulas with a new set variable $G$ (and a set parameter $B$).
For a finite partial function $\sigma\subseteq M\times k\to 2$, we let
\begin{align*}
\sigma_{+}&:=\{\langle i, \E m\,\pi(e,m,G) \rangle: \sigma(e,i)=1\},\\
\sigma_{+,i,\le \ell}&:=\{\langle i, \E m(\pi(e,m,G)\wedge m\le \ell) \rangle: \sigma(e,i)=1\},\\
\sigma_{-}&:=\{\langle i, \E m\,\pi(e,m,G) \rangle: \sigma(e,i)=0\}.
\end{align*}

\begin{defi}[largeness]
Let $\sigma$ be a finite partial function $\sigma\subseteq M\times k\to 2$.
\begin{enumerate}
 \item A Mathias pair $(\bar{F},X)$ is said to be \textit{$\sigma$-large} if for any finite sets of (possibly finite) Mathias pairs $\{(\bar{E}^{t}, Y^{t})\}_{t<s}$ and any bound $\ell'\in M$ such that for all $t<s$ and for all $i<k$, ${E}^{t}_{i}\subseteq \mc{A}_{i}$, $(\bar{E}^{t}, Y^{t})\le (\bar{F}, X)$, $\ell'\ge \max\bar{E}^{t}$, and $X\supseteq \bigsqcup_{t<s}Y^{t}\supseteq X\setminus \ell'$ (\textit{i.e.}, $Y^{t}$'s partition a superset of $X\setminus \ell'$ which is included in $X$),
 there exists $t<s$ such that $(\bar{E}^{t}, Y^{t},\ell')\not\Vdash\langle i,\psi \rangle$ for any $\langle i,\psi \rangle\in \sigma_{+}$ and $Y^{t}$ is not bounded by $\ell'$.
 \item Let $i<k$, $\ell\in M$. Then a Mathias pair $(\bar{F},X\cap \mc{A}_{i})$ is said to be \textit{$\sigma$-large at $i$ up to $\ell$} if the largeness holds for $\sigma_{+,i,\le\ell}$ instead of $\sigma_{+}$ with considering all possible $\Delta^{0}_{2}$-definable sets for $Y^{t}$'s.
Formally, $(\bar{F},X\cap \mc{A}_{i})$ is \textit{$\sigma$-large at $i$ up to $\ell$} if
 for any $\Delta^{0}_{2}$-definable finite sets of Mathias pairs $\{(\bar{E}^{t}, Y^{t})\}_{t<s}$ and any bound $\ell'\in M$ such that for all $t<s$, ${E}^{t}_{i}\subseteq \mc{A}_{i}$, $(\bar{E}^{t}, Y^{t})\le (\bar{F}, X\cap\mc{A}_{i})$, $\ell'\ge \max\bar{E}^{t}$, and $X\cap \mc{A}_{i}\supseteq \bigsqcup_{t<s}Y^{t}\supseteq (X\cap \mc{A}_{i})\setminus \ell'$,
 there exists $t<s$ such that $(\bar{E}^{t}, Y^{t},\ell')\not\Vdash\langle i,\psi \rangle$ for any $\langle i,\psi \rangle\in \sigma_{+,i,\le\ell}$ and $Y^{t}$ is not bounded by $\ell'$.
(Here, we consider all $\Delta^{0}_{2}$-definable sets in $(M,S)$ with any parameters from $S$. Be aware that we do not restrict to $\Delta^{B}_{2}$-sets.)
 %
\end{enumerate}
\end{defi}
Roughly speaking, $\sigma$-largeness guarantees that one can find an extension without forcing any $\langle i,\psi \rangle\in \sigma_{+}$ in the future construction.
\begin{rem}\label{rem:largeness}
\begin{enumerate}
 \item The notion ``$(\bar{F},X)$ is $\sigma$-large'' won't be changed whether we consider Mathias pairs $(\bar{E}^{t}, Y^{t})$ with $Y^{t}$ being a set in the structure or a $\Delta^{0}_{2}$-definable set by Theorem~\ref{thm:Pi01-in-WKL}.3, and it is described by a $\Pi^{B}_{2}$-formula.
 \item For the case ``$(\bar{F},X\cap \mc{A}_{i})$ is $\sigma$-large at $i$ up to $\ell$'', it is essential to consider $\Delta^{0}_{2}$-definable sets, and thus the statement cannot be described by a $\Pi^{B}_{2}$-formula.
In the following construction (which will be $B''$-primitive recursive), we will avoid checking this requirement directly.
\end{enumerate}
\end{rem}

\begin{defi}[condition]
A pre-condition $p=(\bar{F}^{p}, X^{p}, \sigma^{p}, \ell_{0}^{p},\ell_{1}^{p})$ is said to be a \textit{condition} if
\begin{enumerate}
 \item $(\bar{F}^{p}, X^{p})$ is $\sigma^{p}$-large,
 \item $(\bar{F}^{p}, X^{p},\ell_{1}^{p})\Vdash \langle i,\psi \rangle$ for any $\langle i,\psi \rangle\in \sigma^{p}_{-}$,
 \item if $(\bar{F}^{p}, X^{p}\cap \mc{A}_{i})$ is $\sigma^{p}$-large at $i$ up to $\ell_{0}^{p}$, then,
 $\A m\le \ell_{0}^{p}$, $(\bar{F}^{p}, X^{p},\ell_{1}^{p})\Vdash^{+} \langle i, \neg\pi(e,m,G) \rangle$ for any $e\le \ell_{0}^{p}$ with $\sigma^{p}(e,i)=1$.
\end{enumerate}
Define $\mathbb{P}$ as the set of all conditions.
For given two conditions $p,q\in \mathbb{P}$, $q$ properly extends $p$ ($p\succ q$) if
$$(\bar{F}^{p}, X^{p})\ge(\bar{F}^{q}, X^{q})\wedge \ell_{1}^{p}\le \ell_{0}^{q}\wedge \sigma^{p}\subseteq\sigma^{q}.$$
\end{defi}

For a given condition $p=(\bar{F}^{p}, X^{p}, \sigma^{p}, \ell_{0}^{p},\ell_{1}^{p})$, we want to find an extension of $p$.
For this, we introduce a weaker version of the largeness notion.
\begin{defi}[fairness]
Let $\sigma$ be a finite partial function $\sigma\subseteq M\times k\to 2$.
A Mathias pair $(\bar{F},X)$ is said to be \textit{$\sigma$-fair} if
\begin{itemize}
 \item[$(\dag)$] there exist a finite set of Mathias pairs $\{(\bar{E}^{t}, Y^{t})\}_{t<s}$ and a bound $\ell'\in M$ such that ${E}^{t}_{i}\subseteq \mc{A}_{i}$, $(\bar{E}^{t}, Y^{t})\le (\bar{F}, X)$, $\ell'\ge \max\bar{E}^{t}$, $X\supseteq \bigsqcup_{t<s}Y^{t}\supseteq X\setminus \ell'$ such that for any $t<s$,
\begin{itemize}
 \item if $(\bar{E}^{t}, Y^{t},\ell')\not\Vdash\langle i,\psi \rangle$ for any $\langle i,\psi \rangle\in \sigma_{+}$, then $(\bar{E}^{t}, Y^{t},\ell')\Vdash\langle i,\psi \rangle$ for every $\langle i,\psi \rangle\in \sigma_{-}$, or,
 \item $Y^{t}$ is bounded by $\ell'$,
\end{itemize}
 and,
 \item[$(\dag\dag)$]  for any finite set of Mathias pairs $\{(\bar{E}^{t}, Y^{t})\}_{t<s}$ and a bound $\ell'\in M$ which witness the condition $(\dag)$,
 there exists $t<s$ such that $(\bar{E}^{t}, Y^{t},\ell')\not\Vdash\langle i,\psi \rangle$ for any $\langle i,\psi \rangle\in \sigma_{+}$ and $Y^{t}$ is not bounded by $\ell'$.
\end{itemize}
Note that ``$(\bar{F},X)$ is $\sigma$-fair'' can be described by a boolean combination of $\Sigma^{B}_{2}$ and $\Pi^{B}_{2}$ formulas.
\end{defi}

\begin{lem}[$\WKLo+\BIII$]\label{lem:fair-extension}
Let $p=(\bar{F}^{p}, X^{p}, \sigma^{p}, \ell_{0}^{p},\ell_{1}^{p})$ be a condition, and let $\ell'\ge \ell_{1}^{p}$.
Then $(\bar{F}^{p}, X^{p})$ is $\tau$-fair for some $\tau:\ell'\times k\to 2$ extending $\sigma^{p}$.
Moreover, one can find a lexicographically maximal such $\tau$. 
\end{lem}
\begin{proof}
Since $p$ is a condition, $(\bar{F}^{p}, X^{p})$ is $\sigma^{p}$-fair.
We will see by $\Sigma^{0}_{2}$-induction that for any finite set $H\subseteq M\times k$, there exists $\tau:\dom(\sigma^{p})\cup H\to 2$ such that $\tau\supseteq \sigma^{p}$ and $(\bar{F}^{p}, X^{p})$ is $\tau$-fair.
For this, we only need to see that for any $\sigma'$ extending $\sigma^{p}$ such that $(\bar{F}^{p}, X^{p})$ is $\sigma'$-fair and $(e_{0},i_{0})\in M\times k\setminus \dom(\sigma')$, either $\sigma'\cup\{(e_{0},i_{0},0)\}$ or $\sigma'\cup\{(e_{0},i_{0},1)\}$ satisfies the fairness condition for $(\bar{F}^{p}, X^{p})$.
Assume that $(\bar{F}^{p}, X^{p})$ is not $\sigma'\cup\{(e_{0},i_{0},1)\}$-fair.
Since any finite set of Mathias pairs and a bound which witness the condition $(\dag)$ for $(\bar{F}^{p}, X^{p})$ to be $\sigma'$-fair actually witness $(\dag)$ for $(\bar{F}^{p}, X^{p})$ to be $\sigma'\cup\{(e,i,1)\}$-fair, the condition $(\dag\dag)$ for $(\bar{F}^{p}, X^{p})$ to be $\sigma'\cup\{(e_{0},i_{0},1)\}$-fair must fail.
Thus, there exist a finite set of Mathias pairs $\{(\bar{E}^{t}, Y^{t})\}_{t<s}$ and a bound $\ell'\in M$ which witness the condition $(\dag)$ for $\sigma'\cup\{(e,i,1)\}$
such that for any $t<s$, $(\bar{E}^{t}, Y^{t},\ell')\Vdash\langle i,\psi \rangle$ for some $\langle i,\psi \rangle\in \sigma'_{+}\cup\{\langle i_{0}, \E m\,\pi(e_{0},m,G) \rangle\}$ or $Y^{t}$ is bounded by $\ell'$.
Thus, for any $t<s$, if $Y^{t}$ is not bounded by $\ell'$, then $(\bar{E}^{t}, Y^{t},\ell')\not\Vdash\langle i,\psi \rangle$ for any $\langle i,\psi \rangle\in \sigma'_{+}$ implies $(\bar{E}^{t}, Y^{t},\ell')\Vdash\langle i,\psi \rangle$ for any $\langle i,\psi \rangle\in \sigma'_{-}\cup\{\langle i_{0}, \E m\,\pi(e_{0},m,G) \rangle\}$.
This means $\{(\bar{E}^{t}, Y^{t})\}_{t<s}$ and $\ell'$ witness the condition $(\dag)$ for $(\bar{F}^{p}, X^{p})$ to be $\sigma'\cup\{(e_{0},i_{0},0)\}$-fair.
The condition $(\dag\dag)$ for $\sigma'\cup\{(e_{0},i_{0},0)\}$ is automatically satisfied by the same condition for $\sigma'$.
\end{proof}

\begin{lem}[$\WKLo+\BIII$]\label{lem:left-most-extension}
For any $p\in\mathbb{P}$, there exists $q\in \mathbb{P}$ such that $q\prec p$.
Moreover, one can construct such an extension in a ``left-most'' way, \textit{i.e.}, there is a canonical definable way to choose needed elements in the construction of an extension.
\end{lem}
\begin{proof}
For a given condition $p=(\bar{F}^{p}, X^{p}, \sigma^{p}, \ell_{0}^{p},\ell_{1}^{p})\in \mathbb{P}$, put $\ell_{0}=\ell_{1}^{p}$.
By Lemma~\ref{lem:fair-extension}, there exists a lexicographically maximal $\tau:\ell_{0}\times k\to 2$ which extends $\sigma^{p}$ such that $(\bar{F}^{p}, X^{p})$ is $\tau$-fair.
Then one can find a family of low Mathias pairs $\{(\bar{E}^{t}, Y^{t})\}_{t<s}$ (of smallest index) and a bound $\ell'\in M$ which witness $(\dag)$.
By $(\dag\dag)$, pick the smallest $t<s$ such that $(\bar{E}^{t}, Y^{t})$ is $\tau$-large.
Such a $t<s$ exists by $\BII$ since for any $\ell''\ge \ell'$ and for any $\{(\bar{D}^{t}, Z^{t})\}_{t<s''}$ which refines $\{(\bar{E}^{t}, Y^{t})\}_{t<s}$, one can apply $(\dag\dag)$ for $\{(\bar{D}^{t}, Z^{t})\}_{t<s''}$ and $\ell''$.
Note that $\tau$-largeness implies that $(\bar{E}^{t}, Y^{t},\ell')\not\Vdash\langle i,\psi \rangle$ for any $\langle i,\psi \rangle\in \tau_{+}$ and $Y^{t}$ is infinite, thus, by $(\dag)$, $(\bar{E}^{t}, Y^{t},\ell')\Vdash\langle i,\psi \rangle$ for any $\langle i,\psi \rangle\in \tau_{-}$.

Now $(\bar{E}^{t}, Y^{t},\tau,\ell_{0},\ell')$ satisfies the first and second clauses to be a condition.
For the third clause, we use the following claims.
We say that $(\bar{D}', Z')$ is a finite extension of $(\bar{D}, Z)$ at $i$ if $(\bar{D}', Z')\le(\bar{D}, Z)$, $Z\setminus Z'$ is finite, and $D'_{i'}=D_{i'}$ for any $i'\neq i$.
One can observe that finite extensions preserve $\tau$-largeness.
\begin{claim*}
Let $(\bar{D}, Z)$ be a finite extension of $(\bar{E}^{t}, Y^{t})$ at $i$.
If $(\bar{D}, Z\cap \mc{A}_{i})$ is $\tau$-large at $i$ up to $\ell_{0}$, then for any $e< \ell_{0}$ such that $\tau(e,i)=1$,
there exists a finite extension $(\bar{D}', Z')\le(\bar{D}, Z)$ at $i$ such that $D'_{i}\in (D_{i},Z\cap\mc{A}_{i})$ and $(\bar{D}', Z',\max\bar{D}'\cup\{\ell'\})\Vdash^{+} \langle i, \A m\le \ell_{0}\,\neg\pi(e,m,G) \rangle$.
\end{claim*}
\begin{claim*}
If $(\bar{E}^{t}, Y^{t}\cap\mc{A}_{i})$ is $\tau$-large at $i$ up to $\ell_{0}$, then there exists a finite extension $(\bar{D}', Z')\le(\bar{E}^{t}, Y^{t})$ at $i$ such that $D'_{i}\in ({E}^{t}_{i}, Y^{t}\cap\mc{A}_{i})$ and $(\bar{D}', Z',\max\bar{D}'\cup\{\ell'\})\Vdash^{+} \langle i, \A m\le \ell_{0}\,\neg\pi(e,m,G) \rangle$ for all $e<\ell_{0}$ such that $\tau(e,i)=1$.
\end{claim*}
One can easily check the first claim by unfolding the definition of $\tau$-largeness at $i$ up to $\ell_{0}$.
Since finite extensions preserve $\tau$-largeness at $i$, the second claim is obtained by applying the first claim repeatedly. (This is possible within $\III$.)

Now we define $(\bar{D}^{*},Z^{*})\le (\bar{E}^{t}, Y^{t})$ as follows.
For each $i<k$, check whether there exists a finite extension $(\bar{D}', Z')\le(\bar{E}^{t}, Y^{t})$ at $i$ such that $D'_{i}\in ({E}^{t}_{i}, Y^{t}\cap\mc{A}_{i})$ and $(\bar{D}', Z',\max\bar{D}'\cup\{\ell'\})\Vdash^{+} \langle i, \A m\le \ell_{0}\,\neg\pi(e,m,G) \rangle$ for all $e<\ell_{0}$ with $\tau(e,i)=1$. (Note that this condition can be expressed by a $\Sigma^{B}_{2}$-formula.)
Put $D^{*}_{i}=D'_{i}$ if such $\bar{D}'$ exists, and put $D^{*}_{i}=E^{t}_{i}$ otherwise.
(More precisely, one can pick minimal such $\bar{D}^{*}$ within $\III$.)
Put $Z^{*}=Y^{t}\setminus[0,\max \bar{D}^{*}]$.
Then, by the second claim, one can observe that for all $i<k$ and $e\le \ell_{0}$, $(\bar{D}^{*}, Z^{*}, \max\bar{D}^{*}\cup\{\ell'\})\Vdash^{+} \langle i, \A m\le \ell_{0}\,\neg\pi(e,m,G) \rangle$ if $(\bar{D}^{*}, Z^{*}\cap \mc{A}_{i})$ is $\tau$-large at $i$ up to $\ell_{0}$ and $\tau(e,i)=1$.
Take the minimal $\ell_{1}$ so that $\ell_{1}$ bounds $\max\bar{D}^{*}\cup\{\ell'\}$ and a code for $Z^{*}$.
Then $q=(\bar{D}^{*}, Z^{*}, \tau, \ell_{0},\ell_{1})$ is the desired extension.
\end{proof}

For a given $p\in \mathbb{P}$, the extension constructed in the proof of Lemma~\ref{lem:left-most-extension} is said to be a \textit{left-most successor of $p$.}
Note that ``$q$ is a left-most successor of $p$'' can be described by a boolean combination of $\Sigma^{0}_{2}$ and $\Pi^{0}_{2}$ formulas.

Let $p_{0}\succ p_{1}\succ \dots$ be the left-most path of $\mathbb{P}$, \textit{i.e.}, $p_{i+1}$ is a left-most successor of $p_{i}$.
More formally, put 
\begin{align*}
\mc{G}&=\{p_{n}: \E \langle p_{j}\mid j\le n \rangle (p_{0}=(\emptyset, \N, \emptyset, 0,1)\wedge \A j<n(p_{j+1}\mbox{ is a left-most successor of }p_{j}))\},\\
 J&=\{n : \E \langle p_{j}\mid j\le n \rangle (p_{0}=(\emptyset, \N, \emptyset, 0,1)\wedge \A j<n(p_{j+1}\mbox{ is a left-most successor of }p_{j}))\}. 
\end{align*}
Both of $J$ and $\mc{G}$ are $\Sigma^{B}_{3}$. 
Note that $J$ may form a proper cut of $M$.
\begin{lem}[$\WKLo+\BIII$]\label{lem:unbdd-G}
 $\mc{G}$ is unbounded, \textit{i.e.}, for any $x\in M$, there exists $p_{i}\in \mc{G}$ such that $\ell_{1}^{p_{i}}>x$.
\end{lem}
\begin{proof}
Assume that $\mc{G}$ is bounded by some $\bar{\ell}\in M$.
Then the first existential quantifier in the definition of $J$ is bounded.
Thus it is defined by a boolean combination of $\Sigma^{B}_{2}$ and $\Pi^{B}_{2}$ formulas.
Hence $J$ has a maximal element by $\III$, which contradicts Lemma~\ref{lem:left-most-extension}.
\end{proof}

Thus, $\mc{G}$ is cofinal in $M$.
Our next task is to see that at some $i<k$, the construction of a solution works for any $j\in J$.
If we can find such $i<k$, then $\bigcup_{j\in J} F_{i}^{p_{j}}$ is unbounded in $M$.

For each $j\in J$, put 
\begin{align*}
\eta^{j}:=\{i<k: \A m\le \ell_{0}^{p_{j}} (\bar{F}^{p_{j}}, X^{p_{j}},\ell_{1}^{p_{j}})\Vdash^{+} \langle i, \neg\pi(e,m,G) \rangle\mbox{ for any }e\le \ell_{0}^{p_{j}}\mbox{ with }\sigma^{p_{j}}(e,i)=1\}. 
\end{align*}
Here, $i\in\eta^{j}$ means that the construction for color $i$ is sill working at stage $j\in J$.
Trivially, $\eta^{j}\supseteq\eta^{j'}$ if $j\le j'$.
\begin{lem}[$\WKLo+\BIII$]\label{lem:color-survive1}
$\eta^{j}\neq \emptyset$ for any $j\in J$.
\end{lem}
\begin{proof}
By the definition of the condition, it is enough to show that $(\bar{F}^{p_{j}}, X^{p_{j}}\cap \mc{A}_{i})$ is $\sigma^{p_{j}}$-large at $i$ up to $\ell_{0}^{p_{j}}$ for some $i<k$.
Assume not, then for each $i<k$ there exists a witness $\{(\bar{E}^{t,i}, Y^{t,i})\}_{t<s_{i}}$ so that $(\bar{F}^{p_{j}}, X^{p_{j}}\cap \mc{A}_{i})$ is not $\sigma^{p_{j}}$-large at $i$ up to $\ell_{0}^{p_{j}}$.
Then the union $\{(\bar{E}^{t,i}, Y^{t,i})\}_{t<s_{i},i<k}$ indicates that $(\bar{F}^{p_{j}}, X^{p_{j}})$ is not $\sigma^{p_{j}}$-large by Remark~\ref{rem:largeness}.1, which is a contradiction.
\end{proof}
\begin{lem}[$\WKLo+\BIII$]\label{lem:color-survive2}
There exists $i<k$ such that $i\in \eta^{j}$ for any $j\in J$.
\end{lem}
\begin{proof}
Assume that such $i<k$ does not exist.
Then we have $\A i<k\,\E \bar{\ell}\,\E j\in J(i\notin \eta^{j}\wedge \ell_{1}^{p_{j}}<\bar{\ell})$.
Thus, by $\BIII$, there exists $\ell\in\N$ such that $\A i<k\, \E j\in J(i\notin \eta^{j}\wedge \ell_{1}^{p_{j}}<\bar{\ell})$.
By Lemma~\ref{lem:unbdd-G}, there exists $p_{j'}\in\mc{G}$ such that $\ell_{1}^{p_{j'}}>\bar{\ell}$.
Then $\eta^{j'}=\emptyset$ by the monotonicity of $\eta^{j}$, which contradicts Lemma~\ref{lem:color-survive1}.
\end{proof}

\begin{proof}[\it Proof of Theorem~\ref{thm:main-construction}.]
By Lemma~\ref{lem:color-survive2}, pick a color $i<k$ such that $i\in \eta^{j}$ for every $j\in J$ and put $G:=\bigcup_{j\in J} F_{i}^{p_{j}}$.
Then $G\subseteq \mc{A}_{i}$.
Take $e_{\inf}\in \N$ so that $\A m\E n>m(n\in G)\leftrightarrow \A m\neg\pi(e_{\inf},m,G)$.
Then, for large enough $j\in J$, $\sigma^{p_{j}}(e_{\inf},i)=1$ since ``$G$ is finite'' is never forced by an infinite Mathias pair.
Thus, $G$ is infinite by the third clause of the definition of conditions.
$G$ is $\Delta^{B}_{3}$ since $x\in G\leftrightarrow \E j\in J (\ell_{0}^{p_{j}}>x\wedge x\in F_{i}^{p_{j}})$ and $x\notin G\leftrightarrow \E j\in J (\ell_{0}^{p_{j}}>x\wedge x\notin F_{i}^{p_{j}})$.
For any $e\in \N$, $\A m \neg\pi(e,m,G)$ holds if and only if $\E j\in J(\ell_{0}^{p_{j}}>e\wedge \sigma^{p_{j}}(e,i)=1)$.
Thus, any $\Pi^{B\oplus G}_{2}$-formula is equivalent to a $\Sigma^{B}_{3}$-formula, and hence any $\Sigma^{B\oplus G}_{3}$-formula is equivalent to a $\Sigma^{B}_{3}$-formula.
\end{proof}

\bibliographystyle{plain}
\bibliography{bib-rt2inf}

\end{document}